\documentclass{birkjour}
\usepackage{amsmath, amsthm, amscd, amsfonts, amssymb, graphicx, color}
\usepackage[bookmarksnumbered, colorlinks, plainpages]{hyperref}
\usepackage{mathtools}

\newtheorem{Thm}{Theorem}[section]
\newtheorem{Lem}[Thm]{Lemma}
\newtheorem{Pro}[Thm]{Proposition}
\newtheorem{Cor}[Thm]{Corollary}

\theoremstyle{definition}
\newtheorem{Exa}[Thm]{Example}

\theoremstyle{remark}

\newtheorem{Def}[Thm]{Definition}

\newcommand{\R}{\mathbb{R}}

\newcommand{\al}{\alpha}

\newcommand{\la}{\lambda}

\renewcommand{\phi}{\varphi}

\begin{document}

\title[Some solitons on almost coK\"{a}hler manifolds and AHM]{Nature of Some Solitons on Almost coK\"{a}hler Manifolds 
and Asymptotically Harmonic Manifolds}

\author[P. Ghosh]{Paritosh Ghosh}
\address{Department of Mathematics\\
Jadavpur University\\
Kolkata-700032, India.}
\email{paritoshghosh112@gmail.com}

\author[H. M. Shah]{Hemangi Madhusudan Shah}
\address{Harish-Chandra Research Institute\\
A CI of Homi Bhabha National Institute\\
Chhatnag Road, Jhunsi, Prayagraj-211019, India.}
\email{hemangimshah@hri.res.in}

\author[A. Bhattacharyya]{Arindam Bhattacharyya}
\address{Department of Mathematics\\
Jadavpur University\\
Kolkata-700032, India}
\email{bhattachar1968@yahoo.co.in}


\begin{abstract}
In this research, we study the nature of $\eta$-Einstein and gradient $\eta$-Einstein soliton in the framework of almost coK\"{a}hler manifolds and $(\kappa,\mu)$-almost coK\"{a}hler manifolds. We find some expressions for scalar curvature of the almost coK\"{a}hler manifold admitting $\eta$-Einstein soliton in various cases. We also prove that if a $(\kappa,\mu)$-almost coK\"{a}hler manifold admits a gradient $\eta$-Einstein soliton, then either the manifold is coK\"{a}hler, or $N(\kappa)$-almost coK\"{a}hler, or the soliton is trivial.  We present an example which validates our results. Finally,
we investigate the asymptotically harmonic manifolds
admitting non-trivial Ricci solitons and show that 
they exhibit rigid behavior if for example, scalar 
curvature attains maximum.
\end{abstract}
\maketitle

\section{Introduction}
In recent years, the study of geometric flows draws a significant importance in the field of differential geometry. Catino and Mazzieri \cite{Cat} in 2016 first introduced the notion of Einstein soliton as a generalization of Ricci soliton. A Riemannian manifold $(M^n,g)$ is said to be an {\it Einstein soliton} if for some $\lambda\in\R$, there exists a vector field $V$ such that 
\begin{equation}\label{p2eq1}
     \mathfrak{L}_Vg+2S=(2\lambda+r)g
\end{equation}
holds, where $\mathfrak{L}_V$ is the Lie differention in the direction of the vector field $V$, $S$ is the Ricci tensor satisfies $S(X,Y)=g(QX,Y)$, $Q$ being the Ricci operator, $r$ is the Ricci curvature. $V$ is known as the {\it potential vector field}. The Einstein soliton, denoted by $(M,g,V,\lambda)$, is {\it steady} for $\lambda=0$, {\it expanding} for $\lambda<0$ and {\it shrinking} for $\lambda>0$. Einstein soliton creates some self-similar solutions of the Einstein flow $$\frac{\partial}{\partial t}g+2S=rg.$$

As a generalization of Einstein soliton the notion of $\eta$-Einstein soliton is introduced \cite{Abla} and it is given by
\begin{equation}\label{p2eq2}
     \mathfrak{L}_Vg+2S=(2\lambda+r)g+2\beta\eta\otimes\eta,
\end{equation}
where, $\beta$ is some constant and the tensor product $(\eta\otimes\eta)(X,Y)$ is given by $\eta(X)\eta(Y)$.\\
If the potential vector field $V=\nabla f$ is the gradient of some smooth function $f$, then the soliton is known as {\it gradient $\eta$-Einstein soliton} \cite{Abla}. $f$ is known as the {\it potential function}. From (\ref{p2eq2}), for $V=\nabla f$, we get the equation of gradient $\eta$-Einstein soliton
\begin{equation}\label{p2eq3}
    2\nabla^2f+2S=(2\lambda+r)g+2\beta\eta\otimes\eta.
\end{equation}

Recently many authors studied Ricci solitons on almost coK\"{a}hler manifolds like in \cite{Gang, Wang}. In this research, motivated by those work, we investigate the nature of $\eta$-Einstein solitons and gradient $\eta$-Einstein solitons on almost coK\"{a}hler manifolds and $(\kappa,\mu)$-almost coK\"{a}hler manifolds. In \cite{Peter}, the rigidity of gradient Ricci solitons is studied. Inspired by this work,
among other results,
we show that if there exists a non-trivial shrinking or steady Ricci soliton on asymptotically harmonic manifold, then the manifold is flat if its scalar curvature attains minimum. 

\par
In the Preliminary section, we discuss about almost coK\"{a}hler manifolds and $(\kappa,\mu)$-almost coK\"{a}hler manifolds. In section 2, we study Einstein soliton on almost coK\"{a}hler manifold with potential vector field pointwise collinear with the Reeb vector field and showed that the soliton is expanding, steady or shrinking respectively when $r>\|h\|^2$, $r=\|h\|^2$ or $r<\|h\|^2$. In Section 4, we study the scalar curvature of almost coK\"{a}hler manifold admitting $\eta$-Einstein soliton in various cases. In the last section we investigate $\eta$-Einstein soliton and gradient $\eta$-Einstein soliton on $(\kappa,\mu)$-almost cok\"{a}hler manifold and we prove the following theorem.
\begin{Thm}\label{p2t9}
    Let $(M^{2n+1},g)$, $n>1$, be a $(\kappa,\mu)$-almost coK\"{a}hler manifold which admits a gradient $\eta$-Einstein soliton. Then one of the following cases can happen:
    \begin{enumerate}
        \item $M$ is coK\"{a}hler, [when $\kappa=0$]
        \item $M$ is $N(\kappa)$-almost coK\"{a}hler, [when $\mu=0$]
        \item the potential function is constant, i.e., the soliton is trivial.
    \end{enumerate}
\end{Thm}
\noindent
We conclude the section by giving an example to validate some of the results. Section $6$ is devoted to studying the nature of Ricci solitons on asymptotically harmonic manifolds.

\section{Preliminaries}

A Riemannian manifold $M^{2n+1}$ of dimension $2n+1$ $(n\geq1)$ is said to have an {\it almost contact structure} $(\phi, \xi, \eta)$,
if there exists a $(1,1)$-tensor field $\phi$, a vector field $\xi$ and a global $1$-form $\eta$ such that
\begin{eqnarray}\label{p2eq4}
  \phi^{2}X = -X+\eta(X)\xi,~~~~
  \eta(\xi) = 1
\end{eqnarray}
holds for all $X$ on $M^{2n+1}$. $\xi$ is the {\it Reeb vector field}. The manifold $M^{2n+1}$ with the almost contact structure $(\phi, \xi, \eta)$ is known as the almost contact manifold $M^{2n+1}(\phi, \xi, \eta)$ \cite{Bla}. If a Riemannian metric $g$ on almost contact manifold $M^{2n+1}(\phi, \xi, \eta)$ satisfies 
\begin{equation}\label{p2eq5}
  g(\phi X, \phi Y)=g(X, Y)-\eta(X)\eta(Y),
\end{equation}
for all $X$, $Y$ on $M^{2n+1}$, then $g$ is said to be {\it compatible} with the almost contact structure and $M^{2n+1}(\phi, \xi, \eta, g)$ is called {\it almost contact metric manifold} \cite{Bla}. A {\it fundamental 2-form}, denoted by $\Phi$, defined as $\Phi(X,Y)=g(X, \phi Y)$, for all $X$, $Y$ on $M^{2n+1}$. An almost contact metric manifold is said to be \textit{normal}, if the tensor field $N=[\phi,\phi]+2d\eta\otimes\xi$ vanishes everywhere on the manifold, where $[\phi,\phi]$ is the Nijenhuis tensor of $\phi$ \cite{Bla}. 

An almost contact metric manifold $M^{2n+1}(\phi, \xi, \eta, g)$ is called {\it almost coK\"{a}hler manifold} if both $\eta$ and $\Phi$ are closed, or equivalently $d\eta=0$ and $d\Phi=0$. According to D.E. Blair \cite{Bla}, (almost) coK\"{a}hler manifold and (almost) cosymplectic manifold are the same. We refer \cite{Dack, Ols, Ols1, Perrr, Shar} for more details on almost coK\"{a}hler manifolds. 

If $\nabla\phi=0$ or $\nabla\Phi=0$, or equivalently if the almost coK\"{a}hler manifold is normal, then the manifold is called {\it coK\"{a}hler manifold}.

Consider $(1,1)$-type tensor $h=\frac{1}{2}\mathfrak{L}_\xi\phi$ and $h'=h\circ\phi$. Then from \cite{Ols, Ols1}, both $h$ and $h'$ are symmetric and satisfies the following conditions:
\begin{eqnarray}\label{p2eq6}
    h\xi=0, ~~~~\operatorname{tr}(h)=0=\operatorname{tr}(h'),~~~~h\phi=-\phi h,
\end{eqnarray}
\begin{equation}\label{p2eq7}
    \nabla\xi=h', ~~~~\operatorname{div}\xi=0
\end{equation}
\begin{equation}\label{p2eq8}
    S(\xi,\xi)=-\|h\|^2,
\end{equation}
and
\begin{equation}\label{p2eq9}
(\nabla_X\eta)(Y)-(\nabla_Y\eta)(X)=0.
\end{equation}
The Riemannian curvature tensor $R$ is defined as
\begin{equation}\label{p2eq10}
    R(X,Y)Z=\nabla_X\nabla_YZ-\nabla_Y\nabla_XZ-\nabla_{[X,Y]}Z,
\end{equation}
for all $X$, $Y$, $Z$ on $M^{2n+1}$.

Recently many authors studied a spacial class of contact manifold named as $(\kappa, \mu)$-almost coK\"{a}hler manifold as a generalization of $K$-contact manifold and Sasakian manifold. We would refer \cite{Bla, Dack, Shar} for more studies.

An almost coK\"{a}hler manifold $M^{2n+1}(\phi, \xi, \eta, g)$ is said to be a {\it $(\kappa, \mu)$-almost coK\"{a}hler manifold} if for all $X$, $Y$ on $M^{2n+1}$, the Riemannian curvature tensor $R$ satisfies
\begin{equation}\label{p2e11}
    R(X,Y)\xi=\kappa[\eta(Y)X-\eta(X)Y]+\mu[\eta(Y)hX-\eta(X)hY],
\end{equation}
where $(\kappa,\mu)\in \mathbb{R}^2$.
In other words $\xi$ belongs to the generalised $(\kappa,\mu)$-nullity distribution. We have for $(\kappa,\mu)$-almost coK\"{a}hler manifold
\begin{equation}\label{p2e12}
    Q\xi=2n\kappa\xi,~~~~h^2=\kappa\phi^2.
\end{equation}
We recall a lemma which will be used in our work.
\begin{Lem}\label{lem1}\cite{Wang}
    Let $M^{2n+1}(\phi, \xi, \eta, g)$ be a $(\kappa, \mu)$-almost coK\"{a}hler manifold of dimension greater than 3 with $\kappa < 0$, then the Ricci operator is given by
    \begin{equation}\label{p2e13}
        QX=\mu hX+2n\kappa\eta(X)\xi,
    \end{equation}
where $\kappa$ is a non-zero constant and $\mu$ is a smooth function satisfying $d\mu\wedge\eta= 0$.
\end{Lem}
\noindent
Moreover, the expression for Ricci operator $Q$ for a $3$-dimensional non-coK\"{a}hler $(\kappa,\mu)$-almost coK\"{a}hler manifold can be given by
\begin{equation}\label{p2e14}
    QX=\bigg(\frac{r}{2}-\kappa\bigg)X+\bigg(3\kappa-\frac{r}{2}\bigg)\eta(X)\xi+\mu hX.
\end{equation}
We also need the following lemma afterwards.
\begin{Lem}\label{lem2}\cite{Per}
    An almost contact $3$-manifold is normal if and only if $h=0$.
\end{Lem}
\noindent
This theorem from Yano \cite{Yano} will be useful for our work.
\begin{Thm}\label{lem3}
If a compact Riemannian manifold $M$ of dimension greater than 2 with $r=const$ admits an infinitesimal nonisometric conformal transformation $X: \mathfrak{L}_Xg=2\rho g,~~~\rho\ne0,$ and if $QD\rho=aD\rho$, where $a$ is a constant, then $M$ is isometric to a sphere. 
\end{Thm}

\section{Einstein Solitons on Almost CoK\"{a}hler Manifolds}
In this section, we examine the nature of an Einstein soliton on almost
coK\"{a}hler manifolds. We begin this section with this theorem.

\begin{Thm}\label{p2t1}
   Consider an almost coK\"{a}hler manifold $M^{2n+1}(\phi, \xi, \eta, g)$ admits an Einstein soliton. If the non-zero potential vector field is pointwise collinear with the Reeb vector field, then the soliton is
   \begin{enumerate}
       \item expanding if $r>\|h\|^2$,
       \item steady if $r=\|h\|^2$,
       \item shrinking if $r<\|h\|^2$.
   \end{enumerate}
\end{Thm}
\begin{proof}
    Let $V=\al \xi$, where $\al$ some non-zero smooth function. Taking the covarient derivative of $V$ in the direction of a vector field $X$ and using (\ref{p2eq6}), we obtain
    \begin{eqnarray}\label{p2e15}
        \nabla_XV=\nabla_X(\al\xi)=X(\al)\xi+\al\nabla_X\xi=X(\al)\xi+\al h'X.
    \end{eqnarray}
    After expanding (\ref{p2eq1}), we have
    \begin{equation}\label{p2e16}
        g(\nabla_XV, Y)+g(X, \nabla_YV)+2S(X,Y)=(2\lambda+r)g(X, Y).
    \end{equation}
    For $V=\al \xi$, using (\ref{p2e15}) and the fact that $h'$ is symmetric, we get from the above equation
    \begin{equation}\label{p2e17}
        X(\al)\eta(Y)+Y(\al)\eta(X)+2\al g(h'X,Y)+2S(X,Y)=(2\lambda+r)g(X, Y).
    \end{equation}
    Now consider at each point $p\in M^{2n+1}$, a local $\phi$-bases $\{e_k; 0\leq \kappa\leq 2n\}$ on $T_pM$, the tangent space at $p$. Replacing $X=Y=e_k$ in the last equation, summing over $\kappa$ and using (\ref{p2eq5}) ($\operatorname{tr}(h')=0$), we acquire
    \begin{equation}\label{p2e18}
        \xi(\al)=(2n+1)\lambda+\frac{2n-1}{2}r.
    \end{equation}
    Putting $Y=\xi$ in (\ref{p2e17}) and using the symmetric property of $h'$, we obtain 
    \begin{equation}\label{p2e19}
         X(\al)+\xi(\al)\eta(X)+2S(X,\xi)=(2\lambda+r)\eta(X).
    \end{equation}
    Using (\ref{p2e18}), after some calculations (\ref{p2e19}) gives
    \begin{equation}\label{p2e20}
        X(\al)+2S(X,\xi)=-\bigg[(2n-1)\lambda+\frac{2n-3}{2}r\bigg]\eta(X).
    \end{equation}
    Now, eliminating $X$, using (\ref{p2eq7}) we get 
    \begin{equation}\label{p2e21}
         \xi(\al)=2\|h\|^2-\left[(2n-1)\lambda+\frac{2n-3}{2}r\right].
    \end{equation}
    Compairing (\ref{p2e18}) and (\ref{p2e21}), we obtain
    \begin{equation}\label{p2e22}
        \lambda=\frac{1}{4n}(\|h\|^2-r).
    \end{equation}
    Therefore the soliton is expanding for $r>\|h\|^2$, steady for $r=\|h\|^2$ and shrinking for $r<\|h\|^2$.\\
\end{proof}

\begin{Cor}
Let $M^{2n+1}(\phi, \xi, \eta, g)$ be an almost coK\"{a}hler manifold admitting an Einstein soliton. If the non-zero potential vector field is pointwise collinear with the Reeb vector field and if the Ricci tensor $S=\frac{1}{4n}(\|h\|^2+(2n-1)r)g$, then the soliton is trivial.    
\end{Cor}
\begin{proof}
    The proof follows from the above theorem. Notice that for $S=\frac{1}{4n}(\|h\|^2+(2n-1)r)g$ and $\lambda=\frac{1}{4n}(\|h\|^2-r)$, $\mathfrak{L}_{\al\xi}g=0$, that is, $V=\al \xi$ is Killing. Therefore the soliton is trivial.\\
\end{proof}

\section{$\eta$-Einstein Solitons on Almost coK\"{a}hler Manifolds}
Now we investigate the nature of $\eta$-Einstein soliton on almost coK\"{a}hler manifold. We begin with the following theorem.
\begin{Thm}\label{p2t2}
    Let $\xi$ be an $\eta$-Einstein soliton on almost coK\"{a}hler manifold $(M,g)$ which satisfies $\nabla_X\xi=0$. Then $M$ is $\eta$-Einstein manifold with $\delta r=0$.
\end{Thm}
\begin{proof}
    Equation (\ref{p2eq2}) implies, for any $X,Y\in\chi(M)$,
    \begin{equation}\label{p2e23}
        g(\nabla_XV, Y)+g(X, \nabla_YV)+2S(X,Y)=(2\lambda+r)g(X, Y)+2\beta\eta(X)\eta(Y).
    \end{equation}
    Along $V=\xi$ and since $\nabla_X\xi=0$, we get
    \begin{equation*}
        S(X,Y)=\bigg(\lambda+\frac{r}{2}\bigg)g(X, Y)+\beta\eta(X)\eta(Y),
    \end{equation*}
    showing that $M$ is $\eta$-Einstein. Now tracing this equation we obtain
    $$r=-\frac{2}{2n-1}[(2n+1)\lambda+\beta],$$ which is constant.\\
\end{proof}

Consider the distribution $\mathfrak{D}$ on the almost coK\"{a}hler manifold defined as $\mathfrak{D}=\operatorname{ker}\eta$. The next theorem is all about when the potential vector field lies in the distribution $\mathfrak{D}$.
\begin{Thm}\label{p2t3}
    Let $(M,g)$ be an almost coK\"{a}hler manifold admitting an $\eta$-Einstein soliton. If the potential vector field $V\in\mathfrak{D}$, then the scalar curvature of $M$ can be given by $r=-2(\lambda+\beta+\| h\|^2)$.
\end{Thm}
\begin{proof}
    for $X=Y=\xi$ and using (\ref{p2eq8}), we obtain from (\ref{p2e23})
    \begin{equation}\label{p2e24}
        \nabla_\xi V=\bigg(\lambda+\frac{r}{2}+\beta+\| h\|^2\bigg)\xi.
    \end{equation}
    Since $V\in\mathfrak{D}$, $\eta(V)=0$.\\
    Taking the covariant derivative with respect to $\xi$, we have
    \begin{equation}\label{p2e25}
        (\nabla_\xi\eta)(V)+\eta(\nabla_\xi V)=0.
    \end{equation}
    Making use of (\ref{p2eq9}), we get $(\nabla_V\eta)(\xi)=-\eta(\nabla_\xi V)$, so that using (\ref{p2e24}) and (\ref{p2e25}) we acquire
    \begin{equation*}
        \nabla_\xi V=\nabla_V\xi=\bigg(\lambda+\frac{r}{2}+\beta+\|h\|^2\bigg)\xi.
    \end{equation*}
    Therefore, using (\ref{p2eq7}),
    $$0=\operatorname{div}\xi=\lambda+\frac{r}{2}+\beta+\|h\|^2,$$ which gives our required result.\\
\end{proof}

\begin{Thm}\label{p2t4}
    Let $M$ be an almost coK\"{a}hler manifold which admits an $\eta$-Einstein soliton. If the potential vector field $V$ is pointwise collinear with $\xi$ and if scalar curvature, $r$ is non-positive, then $\lambda>0$. 
\end{Thm}
\begin{proof}
    Let $V=\al\xi$ for some smooth $\al$. Then $$\nabla_XV=(X\al)\xi+\al h'X$$. 
    Using this, we get from (\ref{p2e23}),
    \begin{eqnarray}\label{p2e26}
        (X\al)\eta(Y)+(Y\al)\eta(X)+2\al g(h'X,Y)+2S(X,Y)\nonumber\\
        =(2\la+r)g(X,Y)+2\beta\eta(X)\eta(Y).
    \end{eqnarray}
    Tracing this equation and using (\ref{p2eq6}), we get
    \begin{equation}\label{p2e27}
        \xi(\al)=(2n+1)\bigg(\la+\frac{r}{2}\bigg)+\beta-r.
    \end{equation}
    From (\ref{p2e26}), for $Y=\xi$, we obtain
    \begin{equation*}
        X(\al)+\xi(\al)\eta(X)+2S(X,\xi)=(2\la+r+2\beta)\eta(X).
    \end{equation*}
    Using the value of $\xi(\al)$ from (\ref{p2e27}) in the last equation, we have
    \begin{equation*}
        X(\al)+2S(X,\xi)=\bigg[2\la+2r+\beta-(2n+1)\bigg(\la+\frac{r}{2}\bigg)\bigg]\eta(X),
    \end{equation*}
    so that for $X=\xi$ and using (\ref{p2eq8}), we get
    \begin{equation}\label{p2e28}
        \xi(\al)=2\la+2r+\beta-(2n+1)\bigg(\la+\frac{r}{2}\bigg)+2\|h\|^2.
    \end{equation}
    Comparing the values of the equations (\ref{p2e27}) and (\ref{p2e28}),
    $$\la=\frac{1}{2n}[\|h\|^2-(n-1)r].$$
    Hence, for $r\leq 0$, $\la>0$.\\
\end{proof}

\begin{Thm}\label{p2t5}
    Let the almost coK\"{a}hler manifold admitting $\eta$-Einstein soliton is an Einstein manifold. If $V=\xi$ then the constant curvature of the manifold is given by $$r=-\frac{2(2n+1)}{2n-1}(\la+\beta).$$
\end{Thm}
\begin{proof}
    Consider $S(X,Y)=\frac{r}{2n+1}g(X,Y)$. Then (\ref{p2e23}) implies
    \begin{equation*}
        g(\nabla_XV, Y)+g(X, \nabla_YV)=\bigg[2\lambda+\frac{(2n-1)r}{2n+1}\bigg]g(X, Y)+2\beta\eta(X)\eta(Y).
    \end{equation*}
    Along $V=\xi$, we get
    \begin{equation*}
        2g(h'X, Y)=\bigg[2\lambda+\frac{(2n-1)r}{2n+1}\bigg]g(X, Y)+2\beta\eta(X)\eta(Y).
    \end{equation*}
    Putting $X=\xi$ and observing that $h'\xi=0$, we obtain
    \begin{equation*}
        \bigg[2\lambda+\frac{(2n-1)r}{2n+1}+2\beta\bigg]\eta(Y)=0,
    \end{equation*}
    which gives our required result $$r=-\frac{2(2n+1)}{2n-1}(\la+\beta).$$
\end{proof}

\section{$\eta$-Einstein Solitons on $(\kappa,\mu)$-Almost coK\"{a}hler Manifolds}
In this section, we investigate Einstein and $\eta$-Einstein solitons on $(\kappa,\mu)$-almost coK\"{a}hler manifolds. Here we consider $(\kappa,\mu)$-almost coK\"{a}hler manifolds with $\kappa<0$ and we call such type of manifolds as non-coK\"{a}hler $(\kappa,\mu)$-almost coK\"{a}hler manifolds. 

We define an Einstein soliton to be a {\it contact Einstein soliton} if the potential vector field is along the Reeb vector field, that is, $V=\xi$.\\
Now we proceed with the next result.

\begin{Thm}\label{p2t6}
    Let $M^{2n+1}$ be a non-coK\"{a}hler $(\kappa,\mu)$-almost coK\"{a}hler manifold admitting a contact Einstein soliton. Then the manifolds with non-negative scalar curvature admit an expanding soliton.
\end{Thm}
\begin{proof}
    Substituting $V=\xi$ in (\ref{p2e16}) and using (\ref{p2eq6}), we have
    \begin{equation*}
         g(h'X, Y)+g(X, h'Y)+2S(X,Y)=(2\lambda+r)g(X, Y).
    \end{equation*}
    Using the symmetric property oh $h'$, the above equation becomes
    \begin{equation*}
        g(h'X, Y)+g(QX,Y)=\bigg(\lambda+\frac{r}{2}\bigg)g(X, Y).
    \end{equation*}
    Now using (\ref{p2e13}) of Lemma \ref{lem1}, last equation implies
    \begin{equation*}
         g(h'X, Y)+g(\mu hX,Y)+2n\kappa\eta(X)\eta(Y)=\bigg(\lambda+\frac{r}{2}\bigg)g(X, Y).
    \end{equation*}
    For $Y=\xi$, we get
    \begin{equation*}
         g((h'+\mu h)X, \xi)=\bigg(\lambda+\frac{r}{2}-2n\kappa\bigg)\eta(X).
    \end{equation*}
    Contracting the above equation we acquire
    \begin{equation}\label{p2e29}
        2n\kappa=\lambda+\frac{r}{2}.
    \end{equation}
    So, $\kappa<0$ implies $\lambda<-\frac{r}{2}$. Thus if $r\geq0$, then $\lambda<0$ and therefore the soliton is expanding.\\
\end{proof}

\begin{Thm}\label{p2t7}
If a non-coK\"{a}hler $(\kappa,\mu)$-almost coK\"{a}hler manifold admits an $\eta$-Einstein soliton with the potential vector field $V$ conformal, Then $V$ is homothetic.
\end{Thm}
\begin{proof}
    Since $V$ is conformal, there exists $\rho\in C^\infty(M)$ such that $\mathfrak{L}_Vg=2\rho g$. Then from (\ref{p2eq2}), we have 
    $$\rho g(X,Y)+g(QX,Y)=\bigg(\lambda+\frac{r}{2}\bigg)g(X,Y)+\beta\eta(X)\eta(Y).$$
    So, for $X=Y=\xi$, we get
    $$\rho=\lambda+\frac{r}{2}+\beta-2n\kappa.$$
    For a non-coK\"{a}hler $(\kappa,\mu)$-almost coK\"{a}hler manifold, $r=2n\kappa=\operatorname{constant}$. Hence, $$\rho=\lambda+\beta-n\kappa=\operatorname{constant}.$$
    Therefore, $V$ is homothetic.\\
\end{proof}
\noindent
Using this theorem and Theorem \ref{lem3}, we can state:
\begin{Cor}
    Let $M$ be a compact non-coK\"{a}hler $(\kappa,\mu)$-almost coK\"{a}hler manifold admitting an $\eta$-Einstein soliton with the potential vector field $V$ an infinitesimal nonisometric conformal transformation. If $\la+\beta\ne nk$, then $M$ is isometric to a sphere.
\end{Cor}

We now want to recall a definition which is required for our next result.
\begin{Def}\cite{Yano2}
A vector field $V$ is said to be torse forming if it satisfies
\begin{equation}\label{p2eq30}
    \nabla_XV=\sigma X+\omega(X)V,
\end{equation}
where, $\sigma$ is some smooth function on $M$ and $\omega$ is a $1$-form. $V$ is called recurrent if $\sigma=0$.
\end{Def}
\begin{Thm}\label{p2t8}
    Let $\xi$ be a torse forming $\eta$-Einstein soliton in non-coK\"{a}hler $(\kappa,\mu)$-almost coK\"{a}hler $3$-manifold with $\mu\ne0$. Then the manifold is coK\"{a}hler.
\end{Thm}
\begin{proof}
    Since $\xi$ is torse forming, $\nabla_X\xi=\sigma X+\omega(X)\xi$, so that $g(\nabla_X\xi,\xi)=0$ implies $\omega=-\sigma\eta$ and so $\nabla_X\xi=\sigma [X-\eta(X)\xi]=h'X.$\\
    From (\ref{p2e23}), for $V=\xi$ we get,
    \begin{equation}\label{p2eq31}
        \sigma g(\phi X, \phi Y)+\mu g(hX, Y)=(\la+\kappa)g(X,Y)+\bigg(\beta+\frac{r}{2}-3k\bigg)\eta(X)\eta(Y).
    \end{equation}
    For $X=hX$, Equation (\ref{p2eq31}) gives
    \begin{equation}\label{p2eq32}
        \sigma^2g(X,\phi Y)+\mu kg(\phi X,\phi Y)+(\la+\kappa)g(hx,Y)=0.
    \end{equation}
    Similarly, for $Y=hY$, Equation (\ref{p2eq31}) gives
     \begin{equation}\label{p2eq33}
       -\sigma^2g(X,\phi Y)+\mu kg(\phi X,\phi Y)+(\la+\kappa)g(hx,Y)=0.
    \end{equation}
    Subtracting Equation (\ref{p2eq33}) from (\ref{p2eq32}), we obtain $\sigma=0$ so that the vector field $\xi$ is recurrent. Also $h'=0$.\\
    For $X=\phi X$, Equation (\ref{p2e14}) gives
    $$Q(\phi X)=\bigg(\frac{r}{2}-\kappa\bigg)\phi X.$$
    Again replacing $X$ by $\phi X$, the last equation implies 
    $$\mu hX=0.$$
    Hence $h=0$, provided $\mu\ne0$. Therefore $M$ is normal by Lemma \ref{lem2}.\\
\end{proof}

\noindent
Now we have come to the proof of the final main result of this section. 

\begin{proof}[ Proof of Theorem \ref{p2t9}:]
    We know that from Poincar\'{e} lemma,
    $$g(X,\nabla_YDf)=g(Y,\nabla_XDf).$$
    If $V$ is gradient of $f$, then from (\ref{p2e23}), we obtain
    $$QY+\nabla_YDf=\bigg(\la+\frac{r}{2}\bigg)Y+\beta\eta(Y)\xi,$$
    so that,
    \begin{eqnarray}\label{p2e34}
        \nabla_X\nabla_YDf=\frac{1}{2}(Xr)Y+\bigg(\la+\frac{r}{2}\bigg)\nabla_XY+\beta\eta(X)Y\xi+\beta\eta(\nabla_XY)\xi\nonumber\\
       +\beta\eta(Y)\nabla_X\xi-(\nabla_XQ)Y-Q(\nabla_XY),
    \end{eqnarray}
    \begin{eqnarray}\label{p2e35}
        \nabla_Y\nabla_XDf=\frac{1}{2}(Yr)X+\bigg(\la+\frac{r}{2}\bigg)\nabla_YX+\beta\eta(Y)X\xi+\beta\eta(\nabla_YX)\xi\nonumber\\
       +\beta\eta(X)\nabla_Y\xi-(\nabla_YQ)X-Q(\nabla_YX),
    \end{eqnarray}
    and
    \begin{eqnarray}\label{p2e36}
        \nabla_{[X,Y]}Df=\big(\la+\frac{r}{2}\big)(\nabla_XY-\nabla_YX)+\beta\eta(\nabla_XY-\nabla_YX)\xi\nonumber\\
        -Q(\nabla_XY-\nabla_YX).
    \end{eqnarray}
    Therefore, using (\ref{p2e34}), (\ref{p2e35}) and (\ref{p2e36}), we get from (\ref{p2eq10}),
    \begin{eqnarray}\label{p2e37}
        R(X,Y)Df=\frac{1}{2}[(Xr)Y-(Yr)X]+\beta[\eta(Y)h'X-\eta(X)h'Y]\nonumber\\
        +(\nabla_YQ)X-(\nabla_XQ)Y.
    \end{eqnarray}
    Contracting over $X$, we acquire
    \begin{equation}\label{p2e38}
        S(Y,Df)=-\frac{2n-1}{2}g(Y,Dr).
    \end{equation}
    Since $M^{2n+1}$ is a almost coK\"{a}hler with $n>1$, we have
    $$S(X,Y)=\mu g(hX,Y)+2n\kappa\eta(X)\eta(Y).$$
    Tracing we get, $r=2n\kappa=constant.$ So from (\ref{p2e38}), $Df$ is an eigenvector of $Q$ with the eigenvalue $0$. Therefore for $Y=Df$, (\ref{p2e13}) gives
    \begin{equation}\label{p2e39}
        \mu hDf+2n\kappa\eta(Df)\xi=0.
    \end{equation}
    Taking the inner product of this relation with $\xi$ gives
    \begin{equation}\label{p2e40}
     2n\kappa\eta(Df)=0,   
    \end{equation}
    which implies either $\kappa=0$ or $\eta(Df)=\xi(f)=0.$\\
    {\bf Case 1.}  When $\kappa=0$ and $\eta(Df)\neq0$, then clearly $M$ is coK\"{a}hler.\\
    {\bf Case 2.}  When $\kappa<0$, that is, $\eta(Df)=0$, (\ref{p2e39}) implies $\mu hDf=0$. That is, either $\mu=0$ or $hDf=0$. When $\mu=0$, clearly $M^{2n+1}$ is $N(\kappa)$-almost coK\"{a}hler manifold.\\
    {\bf Case 3.} For $\mu\ne0$, $hDf=0$. Now taking inner product of (\ref{p2e11}) with $Z$, we get
    \begin{eqnarray}
        R(Z,\xi;X,Y)=g(R(Z,\xi)X,Y)=\kappa\eta(X)g(Y,Z)-\kappa\eta(Y)g(X,Z)\nonumber\\
        +\mu\eta(X)g(hY,Z)-\mu\eta(Y)g(hX,Z).\nonumber
    \end{eqnarray}
    For $X=Df$, $Y=\xi$, the last equation gives
    \begin{equation}\label{p2e41}
        R(X,\xi;Df,\xi)=-g(X,Df).
    \end{equation}
    Also from (\ref{p2e37}), putting $Y=\xi$ and taking the inner product with $\xi$, we obtain
    $$R(X,\xi;Df,\xi)=0.$$
    Comparing this with (\ref{p2e41}), we acquire $g(X,Df)=0=X(f)$. Hence $f$ is constant.\\
\end{proof}
\noindent
From this result, one can state the following corollary.
\begin{Cor}
    Let $(M^{2n+1},g)$ with $n>1$ be a non-coK\"{a}hler $(\kappa,\mu)$-almost coK\"{a}hler manifold admitting a gradient $\eta$-Einstein soliton. Then the potential function is constant, provided $\mu\ne0$.
\end{Cor}

Now we give an example to verify some of the results.
\begin{Exa}
We consider the non-coK\"{a}hler $(\kappa,\mu)$-almost coK\"{a}hler manifold of dimension $3$ from the Example $5.2$ of \cite{Venk} with $\kappa=-\al^2$, $\mu=0$. The Ricci tensor is given by
\begin{eqnarray*}
    S(e_i,e_j)=&-2\al^2;~~~~~i=j,\\
    S(e_i,e_j)=&0; ~~~~~i\neq j.
\end{eqnarray*}
Clearly scalar curvature $r=\sum S(e_i,e_j)=-6\al^2$.
Take the potential vector field $V=\xi=e_1$. Then from (\ref{p2eq2}), we obtain 
\begin{eqnarray*}
    S(e_1,e_1)=&\la+\frac{r}{2}+\beta,\\
    S(e_2,e_2)=&\la+\frac{r}{2}=S(e_3,e_3),\\
    S(e_i,e_j)=&0; ~~~~i\neq j.
\end{eqnarray*}
Therefore, $\la+\frac{r}{2}=-2\al^2$ and $\beta=0$. Clearly $2n\kappa=-2\al^2=\la+\frac{r}{2}$ validating Equation (\ref{p2e29}) of Theorem \ref{p2t6}.\\
\end{Exa}

\section{Rigidity of Asymptotically Harmonic Manifolds admitting Ricci Solitons}
In this section, we exhibit the rigidity of
asymptotically harmonic manifold admitting Ricci 
soliton. The results here are inspired by the work
of \cite{Peter}. We begin with the defintion of an
asymptoically harmonic manifold.\\

Let $(M,g)$ be a complete manifold without conjugate points. Then the universal covering space $(\Tilde{M},g)$ is simply-connected, complete and without conjugate points. Therefore, by the Cartan-Hadamard theorem $\Tilde{M}$ is diffeomorphic to $\mathbb{R}^n$. Thus it is non-compact.\\

\noindent
 {\bf Busemann function:}
 Let $\gamma_{v}$ be a geodesic line in $\Tilde{M}$, then the two {\it Busemann functions}
 associated to $\gamma_{v}$ are defined as  \cite{P Petersen}:
 $$b_{v}^{+}(x) = \lim_{t \rightarrow \infty} d(x, \gamma_{v}(t)) - t,$$
 $$ b_{v}^{-}(x) = \lim_{t \rightarrow -\infty} d(x, \gamma_{v}(t)) + t.$$

\vspace{0.15in}

 \noindent
 {\bf Asymptotically harmonic manifold:}
$(M, g)$ is called {\it asymptotically harmonic} manifold if $\Tilde{M}$ is so, that is, the mean curvature of all the horospheres of $\Tilde{M}$ is a constant. Equivalently, $\Delta {b^{\pm}_{v}}\equiv h\geq 0$, $\forall v\in\ S \Tilde{M}$, the unit sphere in  $T_{p}{\Tilde{M}}$. \\\\
Our goal is to study Ricci solitons on asymptotically harmonic manifolds. Note that  as asymptotically harmonic manifolds are by definition simply connected manifolds,  by Poincar\'e lemma, Ricci soliton $X$ is a gradient soliton. Let $X$ be a Ricci soliton on
asymptotically harmonic manifold
$M$, then $X=\nabla f$, for some function $f : M \rightarrow R$.
Therefore the Ricci soliton equation
\begin{equation}
    S+\frac{1}{2}\mathfrak{L}_Xg=\lambda g
\end{equation}
reduces to
\begin{equation}\label{e2}
    S+\nabla^2f=\lambda g.
\end{equation}

\noindent
The simply connected, complete,  non-compact, homogeneous and Einstein asymptotically harmonic 
manifolds are completely classified in Theorem $1.1$ 
of \cite{H.08}.  

\begin{Thm}\label{asyhar}
Let $(M,g)$ be a complete, simply connected, non-compact, homogeneous and asymptotically harmonic Einstein manifold. Then $M$ is flat, 
or rank one symmetric of non-compact type or a  nonsymmetric Damek-Ricci space.
\end{Thm}

\noindent
The following result is well-known,
its proof can be found in \cite{Z}.

\begin{Thm}\label{Casy}
If $M$ is  compact asymptotically harmonic manifold with negative sectional curvature, then M is a rank one locally symmetric space.
\end{Thm}

Also it is known that 
asymptotically harmonic manifold of
constant $h$ have volume entropy 
$h$ (cf. \cite{Z}). Suppose that $X$ is a trivial soliton. Then $S=\lambda g$. Therefore $(M,g)$ is Einstein and asymptotically harmonic. Suppose $\lambda>0$. Then by Bonnet-Myer's theorem $M$ is compact with finite fundamental group and conjugate point occurs at distance $\frac{\pi}{\sqrt{(n-1)\lambda}}$. But, by definition asymptotically harmonic manifolds can't have conjugate points. Hence shrinking solitons can't occur.
Now suppose $\lambda=0$, that is, $S=\lambda g=0$. But simply connected and complete  Ricci flat asymptotically  harmonic manifolds are flat \cite{NPHA}.\\

Therefore, from the above results we can 
conclude that: 

\begin{Pro}\label{eincl}
    If there exists trivial potential vector field $X=0$ corresponding to Ricci soliton on a complete, asymptotically harmonic manifold, then it is Einstein, of constant $\lambda$. If $\lambda>0$, then shrinking soliton does not exist on $M$. If $\lambda=0$, then $\Tilde{M}$ is flat. If 
$\lambda<0$, then $M$ is non-flat manifold of exponential volume growth. In case, $\lambda<0$ and if $\Tilde M$ is homogeneous, then 
 it is rank one symmetric space of non-compact type or Damek-Ricci non-symmetric space. In case, $\lambda<0$ and if $K_M<0$, then $M$ is rank one symmetric of compact type. 
\end{Pro}

\vspace{0.1in}

\noindent
Now suppose that $\nabla f$ is a  non-trivial Killing vector field for the Ricci soliton on asymptotically harmonic manifold $M$. Then $\nabla^2f=0$ with $f$ non-constant.
Therefore, $\nabla f$ is a parallel vector field with $\lVert \nabla f\rVert$ constant. So, $$\frac{1}{2}\Delta\lVert\nabla f\rVert^2=\lVert{\nabla}^2 f\rVert^2-S(\nabla f,\nabla f).$$ 
This implies  that $S(\nabla f,\nabla f)=0$.
As $\nabla f$ is a Killing Ricci soliton, we have from Ricci soliton equation, $S=\lambda g,$ which implies that $\lambda=0$. Therefore, $S=0$, so that $M$ is Ricci flat, hence $\Tilde M$ is flat. And $f=b_v^-$ is potential function on $\Tilde M$ \cite{NPHA}.
\noindent
Thus we have shown that:
\begin{Pro}\label{f}
    If $X$ is a Killing vector field for non-trivial Ricci soliton on asymptotically harmonic manifold $M$, then $\Tilde M$ must be flat, and $f=b_v^-$ is the potential function on $\Tilde M$.
    \end{Pro}
    
\begin{Pro}\label{cl}
Suppose that there exists a non-trivial Ricci soliton on compact asymptotically harmonic manifold $M$, then $M$ is Einstein. If scalar curvature, $r=0$, then $\Tilde{M}$  is flat. If $r<0$ and $K_M<0$, then $M$ is rank one symmetric space of compact type. If $r>0$, Ricci solitons don't exist. 
\end{Pro}
\begin{proof}
Clearly (\ref{e2}) implies 
$r+\Delta f=\lambda n$. Integrating we get, $\lambda=\frac{r}{n}$. Consequently, 
$\Delta f = 0$, therefore $f$ is constant.   
Again (\ref{e2}) implies that $S = \lambda g,$
therefore $M$ is Einstein. Now the conclusion follows from Proposition \ref{eincl}.
\end{proof}

\begin{Pro}
    If Ricci soliton on an asymptotically hamronic manifold $M$ is of Gaussian type, then $M$ is compact quotient of flat space  and  $f=\frac{\lambda}{2}d(x,p)^2$, if $\lambda\neq 0$.
\end{Pro}
\begin{proof}
By hypothesis, 
 $\nabla^2f=(\lambda-r)g$, for $\lambda \neq 0$. The result follows from \cite{NPHA}.
\end{proof}

\noindent
Our next results are inspired by  
\cite{Peter}.

\begin{Pro}
If $(M,g)$ admits a gradient Ricci soliton, then
$$\nabla r=2S(\nabla f).$$
\end{Pro}
\begin{proof}
Tracing  (\ref{e2}) and differentiating we obtain,
    $\nabla r=-\nabla\Delta f$. We also have $\nabla r=2\operatorname{div}S$. Therefore, $\nabla r=-2\operatorname{div}\nabla^2f$. 
    But by Bochner's formula:
    $$\operatorname{div}\nabla^2f=S(\nabla f)+\nabla\Delta f.$$
    Consequently, $\nabla r=2S(\nabla f)$.
\end{proof}

\begin{Pro}\label{non-pos}
Let $(M,g)$ be an asymptotically harmonic manifold.
Then Ricci curvature of $M$ is non-positive.
\end{Pro}
\begin{proof}
Let $\Tilde M$ be the universal covering space of $M$.
Let $L_{t} = {\nabla}^2 {b_{v}}^{+}$ denote the second fundamental form of horospheres $b_{v}^{-1}(t)$ of
$\Tilde M$.
Then $L_{t}$ satisfies the Riccati equation, that is for
 $x_{t} \in \{\gamma'(t)\}^{\perp}$,
 $$ {L_{t}}' (x_{t}) + {L_{t}}^2(x_{t}) +
 R(x_{t},\gamma'(t))\gamma'(t)= 0.$$
 Tracing the above equation, we obtain that
 $$ \mbox{tr} {L_{t}}^2 = -\mbox{Ricci}(\gamma'(t),\gamma'(t)).$$
 \end{proof}

 \begin{Pro}
     Suppose that $\nabla f(p)\neq 0$ for all $p\in M$, and scalar curvature of asymptotically harmonic manifold $M$ admitting a Ricci soliton is radially constant. Then $M$ is flat.
 \end{Pro}
 \begin{proof}
     $\nabla_{\nabla f}r=2S(\nabla f, \nabla f)=0$, by hypothesis.\\
     Therefore, 
     \begin{equation}\label{e3}
         S\bigg(\frac{\nabla f(p)}{\lVert \nabla f(p)\rVert},\frac{\nabla f(p)}{\lVert \nabla f(p)\rVert}\bigg)=0, \;\forall p\in M.     \end{equation}
     So $M$ can't be asymptotically harmonic manifold of constant $h>0$, as they have strictly negative Ricci curvature. Therefore, $M$ is asymptotically harmonic of constant $h=0$. Thus 
     any Busemann function on $\Tilde M$ satisfies
     $\Delta b^{+}_{v} = 0.$
      As (\ref{e3}) holds, all the radial Ricci curvatures are zero.\\
      Consider the Bochner's formula for $b^{+}_v$.
      $$ \frac{1}{2}\Delta\lVert\nabla {b_v^{+}}\rVert^2=\lVert \nabla^2 {b_v^{+}}\rVert^2+S(\nabla{b_v^{+}},\nabla b^{+}_v)$$ for all $v\in SM$.
      The foregoing equation gives $\lVert \nabla^2 {b_v}^{+}\rVert^2=0$, which implies $\nabla^2 b_v=0$, $\forall v\in SM$. Consequently, by Ricatti equation, $R(X, \nabla b_v)\nabla b_v=0$.
      Hence, $\Tilde{M}$  is flat and consequently $M$
      is a compact quotient of a flat space. 
 \end{proof}
 
\begin{Cor}\label{f} 
Suppose that $M$, satisfying the hypothesis of the above proposition, has constant scalar curvature and admits a Ricci soliton with $(\nabla f)(p) \neq 0$
for any $p \in M$. Then $M$ is flat.
\end{Cor}
\begin{proof}
    As $\nabla_{\nabla f}r=2S(\nabla f, \nabla f)=0$, the result follows from the above proposition.
\end{proof}

\begin{Thm}
Suppose that an asymptotically harmonic manifold $(M,g)$ admits a  Ricci soliton with 
$X=\nabla f$. Then (i) if $|X|$ attains maximum,
(ii) if for $u = |X|^2,~ |du| \in L^1(M)$, then 
the Ricci soliton is trivial. Consequently, 
$(M,g)$ is Einstein and asymptotically harmonic 
of constant $h = 0$ and also $M=N \times \mathbb{R}$.
$M$ is flat if it is homogeneous.
\end{Thm}
\begin{proof}
Suppose that $(M,g)$ admits a non-trivial Ricci soliton with $X=\nabla f$. Then by \cite{Peter}, it satisfies $$\frac{1}{2}\Delta |X|^2=|\nabla X|^2-S(X,X).$$
As $(M,g)$ is asymptotically harmonic manifold, $S\leq0$ (Proposition \ref{non-pos}).
Thus $ u=|X|^2$ is a subharmonic function on $M$.
If $|X|$ attains maximum or if
$\int_M|du|<\infty$, then by Yau \cite{Y}, $u$ is a harmonic function on $M$. Therefore, $|\nabla X|^2=S(X,X)\leq0$, implies $\nabla f$ is a parallel vector field. Hence, $M=N\times \mathbb{R}$.
$M$ is asymptotically harmonic manifold with minimal horospheres. $M$ is flat if $M$ is homogeneous or if $M$ is harmonic and $f=b_v^-$ (Proposition \ref{eincl}).
\end{proof}

\vspace{0.1in}

\noindent
If a manifold admits a gradient Ricci soliton, then the scalar curvature satisfies  \cite{Peter}:
\begin{eqnarray}\label{s}
\frac{1}{2}(\Delta-D_{\nabla f})r&=&\frac{1}{2}\Delta_fr\nonumber\\
&=&-|S|^2+\lambda r\\
&=&-|S-\frac{1}{n}r_g|^2+r(\lambda-\frac{r}{n})\nonumber.
\end{eqnarray}
Clearly $\Delta_fr$ is an elliptic operator.

\begin{Thm}
    If $(M,g)$ is an asymptotically harmonic manifold admitting shrinking or steady soliton with
    $(\nabla f)(p) \neq 0$.  If the
scalar curvature attains the minimum, then it is flat.   \end{Thm}
\begin{proof}
Clearly $\Delta_fr$ is an elliptic operator.
And as $S \leq 0$, $r$ is a $f$-superharmonic 
function on $M$, if  $(M,g)$ admits shrinking or steady soliton by (\ref{s}). Hence, by the minimum principle, $r$  is a $f$-constant function on $M$.
Hence, by $(\ref{s})$ it follows that 
$|S|^2=\lambda r\leq0$, consequently $S \equiv 0$.
Therefore, $M$ is  Ricci flat and hence flat \cite{NPHA}.
\end{proof}

\begin{Pro}
    If $(M,g)$ is an asymptotically harmonic manifold admitting a Ricci soliton and if scalar curvature, $r$  is a $f$-harmonic function, then 
    $n\lambda\leq r\leq 0$. Consequently,
    $\lambda \leq 0$. If $\lambda =0$, then $\Tilde M$ is flat. And if $n\lambda\leq r$, then 
    $M$ is Einstein of constant $\lambda< 0$. The 
    conclusion of Proposition \ref{eincl} holds.
    \end{Pro}
\begin{proof}
    The proof follows as in this case by (\ref{s}),
 $|S-\frac{r_g}{n}|^2=r(\lambda-\frac{r}{n})$ and $r\leq 0$.
\end{proof}

\noindent
{\bf Conclusion:} Our results generalize the results 
of \cite{NPHA} for harmonic manifolds which are asymptotically harmonic manifolds. The results obtained here exhibit the rigidity of asymptotically harmonic manifolds admitting a non-trivial Ricci soliton.

\section{Acknowledgements}
The corresponding author, Mr. Paritosh Ghosh, thanks UGC Junior Research Fellowship of India. The authors also like to thank Dr. Sumanjit Sarkar for his wishful help in this research.

\end{document}